\documentstyle[amssymb,amsfonts,12pt]{amsart}

\newtheorem{te}{Theorem}[section]

\newtheorem*{ack*}{Acknowledgment}

\def\N{{\mathbb N}}\def\les{{\;\lessapprox}\;}

\def\Qc{{\mathcal Q}}
\def\Z{{\mathbb Z}}
\def\P{{\mathbb P}}
\def\nint{\mathop{\diagup\kern-13.0pt\int}}

\newenvironment{proof}{\noindent {\bf Proof} }{\endprf\par}
\def \endprf{\hfill  {\vrule height6pt width6pt depth0pt}\medskip}
\def\emph#1{{\it #1}}

\begin{document}
\author{Jean Bourgain}
\address{Department of Mathematics, Princeton University}
\email{bourgain2010@@gmail.com}
\author{Ciprian Demeter}
\address{Department of Mathematics, Indiana University,  Bloomington IN}
\email{demeterc@@indiana.edu}
\thanks{The authors are  partially supported by the Collaborative Research NSF grant DMS-1800305}

\title[On $k$th powers in arithmetic progression]{On the number of $k$th powers inside arithmetic progressions}

\begin{abstract}
We find upper bounds that are sharp  for the number of $k$th powers inside arbitrary arithmetic progressions whose step has $O(1)$ many divisors.
\end{abstract}
\maketitle

For fixed $k\ge 2$, how many kth powers of integers $t^k$ can lie inside an arithmetic progression of length $N$? By considering the progression $\{1,2,\ldots, N\}$ we see that this number can be as large as $\sim N^{\frac1k}$. In this note we search for upper bounds, and show that they  match the above lower bound  in the special case when the step of the progression has $O(1)$ many divisors.
\medskip

More precisely, let $\Qc_k(N;q,a)$ denote the number of $k$th powers  in the arithmetic progression $a+q,a+2q,\ldots,a+Nq$. Write
$$\Qc_k(N)=\sup_{a,q\in\N\atop{q\not=0}}\Qc_k(N;q,a).$$
Rudin \cite{Ru} has conjectured that $\Qc_2(N)\sim N^{\frac12}$. We may similarly conjecture that  for each $k\ge 2$
\begin{equation}
\label{ucuycfiedoie}
\Qc_k(N)\les N^{\frac1k},
\end{equation}
where $\les$ denotes logarithmic losses of the form $(\log N)^{O(1)}$. The logarithmic loss here is added for extra safety, it is not clear whether it is really needed.

 The best known upper bound for $\Qc_2(N)$ is due to Bombieri and Zannier \cite{BZ}
$$\Qc_2(N)\les N^{\frac35}.$$
This builds on earlier work \cite {BGP} of Bombieri, Granville and Pintz that proved the result with exponent $\frac23$ in place of $\frac35$. These rely on deep results in number theory regarding rational points on  curves.

The papers \cite{Gra} and \cite{CG} contain a nice discussion on the problem.
In particular, they mention a refined version of  Rudin's conjecture, according to which the progression $\{24n+1:\;0\le n\le N-1\}$ contains the largest number of squares.
\smallskip

We verify the conjectured inequality \eqref{ucuycfiedoie} in the case when the step $q$ of the progression has $O(1)$ divisors. In fact we observe that the argument extends to arbitrary polynomials. It would be of  interest to lower the exponent of $d(q)$ as much as possible.

\begin{te}
\label{Bonumbersquares}
Let $d(q)$ be the number of divisors of $q$. Then for each polynomial $P_k$ of degree $k\ge 1$ with integer coefficients and each $a,q,N\in\Z$ we have   
$$|\{t\in \Z:\;P_k(t)\in\{a+q,a+2q,\ldots,a+Nq\}\}|\lesssim d(q)^{k-1}N^{\frac1k}.$$
The implicit constant depends only on $k$.
\end{te}
\begin{proof}
We use induction on $k$. The case $k=1$ is trivial. Assume the statement holds for $k$, and let $P_{k+1}$ be a polynomial of degree $k+1$. Fix $t_0\not=t$ such that
$$P_{k+1}(t),P_{k+1}(t_0)\in\{a+q,a+2q,\ldots,a+Nq\}.$$
Write $$P_{k+1}(t)-P_{k+1}(t_0)=(t-t_0)P_k(t),$$
for some polynomial $P_k$ of degree $k$.
We must have
\begin{equation}
\label{wqwdiouc7rcyf9ex}
\begin{cases}(t-t_0)=n_1q_1\\P_k(t)=q_2n_2\end{cases}
\end{equation}
with  $q_1q_2=q$. Moreover $n_1n_2\le N$. We must have that either $$n_1\le N^{\frac1{k+1}}$$
or 
$$n_2\le N^{\frac{k}{k+1}}.$$
Let us fix the pair $(q_1,q_2)$. In the first case, there are $\le N^{\frac1{k+1}}$ possible values of $t$ (considering only the first equation in \eqref{wqwdiouc7rcyf9ex}), while in the second case there are $O(d(q_2)^{k-1}N^{\frac1{k+1}})$ values of $t$, due to the induction hypothesis (considering only the second equation in \eqref{wqwdiouc7rcyf9ex}). Since there are $d(q)$ many ways to choose the pair $(q_1,q_2)$ we conclude that the total contribution is 
$$\lesssim d(q)(N^{\frac1{k+1}}+d(q)^{k-1}N^{\frac1{k+1}})\lesssim d(q)^kN^{\frac1{k+1}}.$$
  
\end{proof}

\begin{ack*}
 This argument was discovered after Igor Shparlinski and Houcein el Abdalaoui have informed us of an alternative, equally simple argument. We are grateful to them for alerting us of their argument and for helpful discussions. 
\end{ack*}

\end{document}